\documentclass[11pt,reqno]{amsart}
\usepackage[a4paper, hmargin={2.7cm,2.7cm},vmargin={3.3cm,3.3cm}]{geometry}
\usepackage[english]{babel}
\usepackage{color}
\usepackage[noadjust]{cite}
\usepackage{amsmath}
\usepackage{amssymb}
\usepackage{amsfonts}
\usepackage{dsfont}
\usepackage{amsthm}
\usepackage{enumerate}
\usepackage{amsthm}

\usepackage{etoolbox}

\numberwithin{equation}{section}

\makeatletter
\patchcmd{\ttlh@hang}{\parindent\z@}{\parindent\z@\leavevmode}{}{}
\patchcmd{\ttlh@hang}{\noindent}{}{}{}
\makeatother

\theoremstyle{plain}
\newtheorem{theorem}{Theorem}[section]
\newtheorem{lemma}[theorem]{Lemma}
\newtheorem{proposition}[theorem]{Proposition}

\theoremstyle{definition}

\theoremstyle{remark}
\newtheorem{remark}[theorem]{Remark}

\def\XXint#1#2#3{{\setbox0=\hbox{$#1{#2#3}{\int}$ }
\vcenter{\hbox{$#2#3$ }}\kern-.6\wd0}}

\DeclareMathOperator{\ind}{ind}
\DeclareMathOperator{\capad}{Ad}

\DeclareMathOperator{\pker}{pker}

\DeclareMathOperator{\covol}{covol}
\DeclareMathOperator{\conv}{conv}

\newcommand{\PH}{\mathrm{P}(\mathcal{H}_{\pi})}
\newcommand{\PHs}{\mathrm{P}(\mathcal{H}_{\pi}^{\infty})}

\newcommand{\Hpi}{\mathcal{H}_{\pi}}
\newcommand{\Hs}{\mathcal{H}_{\pi}^{\infty}}

\title{Completeness of coherent state subsystems for nilpotent Lie groups}

\author{Jordy Timo van Velthoven}

\address{Delft University of Technology,
Mekelweg 4, Building 36,
2628 CD Delft, The Netherlands.}
\email{j.t.vanvelthoven@tudelft.nl}

\makeatletter
\@namedef{subjclassname@2020}{\textup{2020} Mathematics Subject Classification}
\makeatother


\begin{document}

\maketitle

\begin{abstract}
Let $G$ be a nilpotent Lie group and let $\pi$ be a coherent state representation of $G$.
The interplay between the cyclicity of the restriction $\pi|_{\Gamma}$ to a lattice $\Gamma \leq G$ and the completeness of subsystems of coherent states based on a homogeneous $G$-space is considered.
In particular, it is shown that necessary density conditions for Perelomov's completeness problem can be obtained
via density conditions for the cyclicity of $\pi|_{\Gamma}$.
\end{abstract}

\section{Introduction}
Let $G$ be a connected unimodular Lie group and let $(\pi, \Hpi)$ be an irreducible unitary representation of $G$. For a unit vector $\eta \in \Hpi$, consider its orbit under the action $\pi$ on $\Hpi$,
\begin{align} \label{eq:coherent_system}
\pi (G) \eta = \big\{ \pi (g) \eta : g \in G \big\}.
\end{align}
As $\pi$ is irreducible,  $\pi(G) \eta$ is complete in $\Hpi$. Two elements $\pi (g_1) \eta$ and $\pi(g_2) \eta$ differ from one another up to a phase factor, i.e. determine the same state or ray, only if $\pi (g_2^{-1} g_1) \eta \in \mathbb{C}  \eta$.

Let $H \leq G$ be a closed subgroup that stabilises the state defined by $\eta \in \Hpi$, i.e.
\begin{align} \label{eq:stabiliser}
\pi(h) \eta = \chi(h) \eta, \quad h \in H,
\end{align}
where $\chi : H \to \mathbb{T}$ is a unitary character of $H$. Denote by $X = G/H$ the associated homogeneous $G$-space
and let $\sigma : X \to G$ be a cross-section for the canonical projection $p : G \to X$. Then the system of coherent vectors
\begin{align} \label{eq:coherent_state}
\{ \eta_x \}_{x \in X} = \{ \pi (\sigma(x)) \eta \}_{x \in X},
\end{align}
determine a \emph{$\pi$-system of coherent states based on $X$}, in the sense of \cite{perelomov1972coherent, rawnsley1977coherent}.

It will be assumed that $X = G/H$ is unimodular, i.e. $X$ admits a $G$-invariant positive Radon measure $\mu_X$, and that
$\eta$ is \emph{admissible}, that is,
\begin{align} \label{eq:admissible_intro}
\int_X | \langle \eta, \eta_x \rangle |^2 \; d\mu_X (x) < \infty.
\end{align}
Then there exists an admissibility constant $d_{\pi, \eta} > 0$ such that
\begin{align} \label{eq:resolutionID}
 \int_X | \langle f , \eta_x \rangle |^2 \; d\mu_X (x) = d_{\pi, \eta}^{-1} \| f \|^2_{\Hpi}, \quad \text{for all} \quad f \in \Hpi.
\end{align}
The identity \eqref{eq:resolutionID} implies, in particular, that the system \eqref{eq:coherent_state} is overcomplete, i.e.  the system $\{\eta_x \}_{x \in X}$ contains proper subsystems which are complete in $\Hpi$.

For an irreducible representation $(\pi, \Hpi)$ of $G$ that is square-integrable modulo the center $Z = Z(G)$ (resp. the kernel $K = \ker(\pi)$), any vector $\eta \in \Hpi $ satisfies \eqref{eq:stabiliser} and \eqref{eq:admissible_intro} for $H = Z$ (resp. $H = K$). Another common choice \cite{rawnsley1977coherent, perelomov1986generalized, kostant1982symplectic, Odzijewicz1992coherent} for the index space $X = G / H$ is a symplectic $G$-space or a homogeneous K\"ahler manifold that arises as a phase space in geometric quantization \cite{woodhouse1992geometric}. Subgroups $H \leq G$ defining such a phase space do not need to satisfy \eqref{eq:stabiliser} for all $\eta \in \Hpi$ and might not be contained in the isotropy group of a chosen $\eta$.

In \cite{perelomov1972coherent, perelomov1986generalized}, a particular focus is on coherent states for which the stabilising subgroup
$H \leq G$ is assumed to be maximal with the property \eqref{eq:stabiliser}, that is, $H = G_{[\eta]}$, where
\begin{align} \label{eq:projective_stabiliser}
G_{[\eta]} := \big\{ g \in G : \pi (g) \eta \in \mathbb{C} \eta \big\}
\end{align}
 is the stabiliser of $\eta$ for the $G$-action in the projective Hilbert space $\PH$.
 The associated coherent states are so-called \emph{Perelomov-type coherent states}; see Section~\ref{sec:perelomov}.

 Perelomov's completeness problem \cite{perelomov1972coherent, perelomov1986generalized} concerns the completeness of subsystems arising from discrete subgroups $\Gamma \leq G$ for which the volume of $\Gamma \backslash X$ is finite.
More explicitly, subsystems parametrised by an orbit $\Gamma' := \Gamma \cdot o$ of the base point $o := e H\in X$,
\begin{align} \label{eq:lattice_orbits}
\{ \eta_{\gamma'} \}_{\gamma' \in \Gamma'} = \{ \pi (\sigma(\gamma')) \eta \}_{\gamma' \in \Gamma'}.
\end{align}
Criteria for the completeness of subsystems  \eqref{eq:lattice_orbits} involving the volume of the coset space $\Gamma \backslash X$ and the admissibility constant $d_{\pi, \eta} > 0$
were posed as a problem in \cite[p.226]{perelomov1972coherent} and \cite[p.44]{perelomov1986generalized}. Note that if $H = G_{[\eta]}$, then $X = G/ G_{[\eta]}$ depends on $\eta$, and so does the volume of $\Gamma \backslash G / G_{[\eta]}$.

The classical example of coherent states arises from the Heisenberg group $G = \mathbb{H}^1$ and a Schr\"odinger representation $(\pi, L^2 (\mathbb{R}))$ of $\mathbb{H}^1$. For any $\eta \in L^2 (\mathbb{R}) \setminus \{0\}$, the stabiliser $G_{[\eta]}$ defined in \eqref{eq:projective_stabiliser} coincides with the centre $Z(\mathbb{H}^1)$ of $\mathbb{H}^1$, and $X = G/ G_{[\eta]} \cong \mathbb{R}^2$. Therefore, the coherent state system \eqref{eq:coherent_state} is parametrised by the classical phase space $\mathbb{R}^2$ and the subsystem \eqref{eq:lattice_orbits} associated to $\Gamma \subset \mathbb{H}^1$ is parametrised by a lattice $\Gamma' \subset \mathbb{R}^2$. If  $\pi$ is treated as a projective representation $\rho$ of $G / G_{[\eta]} \cong \mathbb{R}^2$, then the coherent vectors \eqref{eq:coherent_state} and the subsystem \eqref{eq:lattice_orbits} arise as orbits of $\mathbb{R}^2$ and $\Gamma'$, respectively. In particular, a subsystem $\{ \pi (\sigma(\gamma')) \eta \}_{\gamma' \in \Gamma'}$ is complete in $L^2 (\mathbb{R})$ if, and only if, $\eta$ is a cyclic vector for $\rho|_{\Gamma'}$, i.e. the linear span of $\rho (\Gamma') \eta$ is dense in $L^2 (\mathbb{R})$. Perelomov's completeness problem for the Heisenberg group is therefore equivalent to determining whether a vector is cyclic for the restriction $\rho |_{\Gamma'}$. If $\eta$ is the Gaussian, the cyclicity of $\eta$ has been completely characterised in \cite{perelomov1971remark, bargmann1971on} (see also \cite{neretin2006perelomov}) in terms of the co-volume or density of the lattice.  The necessity of these density conditions have been shown to hold for arbitrary vectors and in arbitrary dimensions \cite{ramanathan1995incompleteness}, but a density condition alone is not sufficient for describing the cyclicity of the Gaussian in higher-dimensions \cite{pfander2013remarks, grochenig2011multivariate}.
The criteria \cite{perelomov1971remark, bargmann1971on,ramanathan1995incompleteness} coincide with the density conditions characterising the cyclicity of the restricted projective representations as obtainted in, e.g. \cite{rieffel1981von, bekka2004square}.

In other settings than the Heisenberg group, the stabilisers $G_{[\eta]}$ defined in \eqref{eq:projective_stabiliser} do not need to be normal subgroups and could depend crucially on the vector $\eta \in \Hpi \setminus \{0\}$.
For example, this occurs for the holomorphic discrete series  $\pi$ of $G = \mathrm{PSL}(2, \mathbb{R})$, where $G_{[\eta]} = \mathrm{PSO}(2)$ for a class of rotation-invariant vectors $\eta$. Hence, the coherent vectors \eqref{eq:coherent_state} do not arise as orbits of a (projective) representation of $G / G_{[\eta]}$ and the subsystems \eqref{eq:lattice_orbits} are not parametrised by an associated discrete subgroup.
Perelomov's  problem for the highest weight vector has been studied for this setting in \cite{perelomov1973coherent, kellylyth1999uniform, jones2020bergman}, and the criteria for the cyclicity of $\pi|_{\Gamma}$ are quite different
from the completeness of coherent state subsystems; see \cite[Section 9.1]{romero2020density} for an overview.

Of particular interest are representations and vectors that support a system of coherent states based on an index manifold $X = G / H$ with additional properties, such as a symplectic \cite{moscovici1978coherent, moscovici1977coherent} or complex structure \cite{lisiecki1990kaehler, neeb1996coherent}.
For nilpotent Lie groups, another common choice (cf.~\cite[Section 10]{perelomov1986generalized}) is the manifold $X$ to be the corresponding coadjoint orbit $\mathcal{O}_{\pi}$ of the representation $\pi$, which forms the classical phase space, like in the special case of the Heisenberg group.

The purpose of this note is to combine characterisations of coherent state representations \cite{moscovici1977coherent, lisiecki1990kaehler, neeb1996coherent} and criteria for the cyclicity of restricted representations \cite{romero2020density, bekka2004square} to obtain necessary density conditions for (variants of) Perelomov's completeness problem on nilpotent Lie groups.

The first result on the completeness of subsystems concerns $\pi$-systems of coherent states
based on the coadjoint orbit $\mathcal{O}_{\pi}$. (cf. Section \ref{sec:prelim} for the precise definitions.)

\begin{theorem} \label{thm:intro}
Let $G$ be a connected, simply connected nilpotent Lie group and let $\Gamma \leq G$ be a discrete, co-compact subgroup. Suppose $(\pi, \Hpi)$ is an irreducible representation of $G$ that admits an admissible vector $\eta \in \Hpi \setminus \{0\}$ defining a $\pi$-system of coherent states based on a homogeneous $G$-space $X = G/H \cong \mathcal{O}_{\pi}$, with admissibility constant $d_{\pi, \eta} > 0$. Then
\begin{enumerate}[(i)]
\item $H = \big\{ g \in G : \pi (g) \in \mathbb{C} \cdot I_{\Hpi} \big\}$;
\item If  $\{ \pi (\sigma (\gamma')) \eta \}_{\gamma' \in \Gamma \cdot o}$ is complete in $\Hpi$, then $\covol(p(\Gamma)) d_{\pi, \eta} \leq 1$.
\end{enumerate}
(The value $\covol(p(\Gamma)) d_{\pi, \eta}$ is independent of the normalisation of $G$-invariant measure on $X$.)
\end{theorem}

Theorem \ref{thm:intro} considers $\pi$-systems of coherent states parametrised by the canonical phase space $\mathcal{O}_{\pi}$ (cf. \cite[Section 10]{perelomov1986generalized}), and provides a necessary condition for the completeness of associated subsystems.
 The representations satisfying the hypothesis of Theorem \ref{thm:intro} are called \emph{coherent state representations} in \cite{moscovici1977coherent}, and are characterised as those being an irreducible representation whose associated coadjoint orbit is a linear variety. The considered representations are therefore essentially square-integrable, like in the special case of the Heisenberg group.

The second result concerns $\pi$-systems of coherent states associated to vectors yielding a symplectic projective orbit (cf. Section~\ref{sec:perelomov} for the precise definitions.)

\begin{theorem} \label{thm:intro2}
Let $G$ be a connected, simply connected nilpotent Lie group and let $\Gamma \leq G$ be a discrete, co-compact subgroup. Suppose $(\pi, \Hpi)$ is an irreducible representation of $G$ that admits an admissible vector $\eta \in \Hpi \setminus \{0\}$ yielding a symplectic orbit and defines a $\pi$-system of coherent states based on $X = G / G_{[\eta]}$, with admissibility constant $d_{\pi, \eta} > 0$.
Then
\begin{enumerate}[(i)]
\item $G_{[\eta]} = \big\{ g \in G : \pi (g) \in \mathbb{C} \cdot I_{\Hpi} \big\}$;
\item If  $\{ \pi (\sigma (\gamma')) \eta \}_{\gamma' \in \Gamma \cdot o}$ is complete in $\Hpi$, then $\covol(p(\Gamma)) d_{\pi, \eta} \leq 1$.
\end{enumerate}
\end{theorem}

In contrast to Theorem \ref{thm:intro}, the index manifold $X = G / G_{[\eta]}$ in Theorem \ref{thm:intro2} is selected via the maximal subgroup \eqref{eq:projective_stabiliser} stabilising the state determined by
 $\eta \in \Hpi \setminus \{0\}$. The vectors $\eta \in \Hpi$ yielding a symplectic orbit
 play a distinguished role in geometric quantization \cite{kostant1982symplectic, Odzijewicz1992coherent}.
 Theorem \ref{thm:intro2} applies, in particular, to smooth vectors of a square-integrable representation (see Proposition \ref{prop:SI_symplectic}) and to so-called \emph{highest weight vectors} (see Remark \ref{rem:complex_orbit}).

 The proofs of Theorem \ref{thm:intro} and Theorem \ref{thm:intro2} are relatively simple and short, but they hinge on a combination of several non-trivial statements on coherent state representations \cite{moscovici1977coherent, lisiecki1990kaehler, neeb1996coherent} and density conditions for restricted discrete series \cite{romero2020density, bekka2004square}. More precisely, exploiting results of \cite{moscovici1977coherent, lisiecki1990kaehler, neeb1996coherent}, it will be shown that
 the completeness of coherent state subsystems is equivalent to the admissible vector being a cyclic vector for a restricted \emph{projective} representation; the necessary density conditions then being a direct consequence of \cite{romero2020density}.

\subsection*{Notation} For a complex vector space $\mathcal{H}$, the notation $\mathrm{P}(\mathcal{H})$ will be used for its projective space, i.e. the space of all one-dimensional subspaces. The subspace or ray generated by $\eta \in \mathcal{H} \setminus \{0\}$ will be denoted by $[\eta] := \mathbb{C} \eta$. Henceforth, unless stated otherwise, $G$ is a connected, simply connected nilpotent Lie group with exponential map $\exp : \mathfrak{g} \to G$. Haar measure on $G$ is denoted by $\mu_G$. If $\Lambda \leq G$ is a discrete subgroup, then the co-volume is defined as
$\covol(\Lambda) := \mu_{G / \Lambda} (G/\Lambda)$, where $\mu_{G/\Lambda}$ denotes $G$-invariant Radon measure on $G/\Lambda$.

\section{Coherent state representations of nilpotent Lie groups} \label{sec:prelim}
This section provides preliminaries on irreducible representations of nilpotent Lie groups and associated coherent states. References for these topics are the books \cite{corwin1990representations} and \cite{ali2014coherent, perelomov1986generalized}.

\subsection{Coadjoint orbits} \label{sec:coadjoint}
Let  $\mathfrak{g}^*$  denote the dual vector space of $\mathfrak{g}$.
The coadjoint representation $\capad^{*} : G \to
\mathrm{GL}(\mathfrak{g}^{*})$ is defined by $\capad^{*} (g) \ell=
\ell\circ\capad(g)^{-1}$ for $g \in G$ and $\ell \in \mathfrak{g}^*$.
The stabiliser of $\ell\in\mathfrak{g}^{*}$ is the connected closed subgroup
$G(\ell) = \{  g\in G: \capad^*(g) \ell=\ell\}  $, its Lie
algebra is the annihilator subalgebra $\mathfrak{g} (\ell)=\{
X\in\mathfrak{g}: \ell([Y, X])
=0,\;\forall Y \in\mathfrak{g}\}  . $

For $\ell\in\mathfrak{g}^{*}$, its \emph{coadjoint orbit} is denoted
by $\mathcal{O}_{\ell} := \capad^*(G) \ell$ and endowed with the relative topology from
$\mathfrak{g}^*$. The orbit $\mathcal{O}_{\ell}$ is homeomorphic to $G / G(\ell)$; in notation: $\mathcal{O}_{\ell} \cong G / G(\ell)$.

\subsection{Irreducible representations}
A Lie subalgebra $\mathfrak{p}$ of $\mathfrak{g}$ is \emph{subordinated} to
$\ell\in\mathfrak{g}^{\ast}$ if $\ell(X) =0$ for
every $X\in [  \mathfrak{p,p}]  $. If $\mathfrak{p}$ is subordinate
to $\ell$, then the map $\chi_{\ell} : \exp( \mathfrak{p}) \to\mathbb{T}$,
$\chi_{\ell}\left(  \exp (X)\right)  =e^{2\pi i\ell(X)} $ defines a unitary character of $P = \exp(\mathfrak{p})$. The associated
induced representation of $G$ is denoted by
$
\pi_{\ell} = \pi(\ell, \mathfrak{p}) = \mathrm{ind}_{P}^{G}\left(  \chi_{\ell
}\right)  .
$

For every $\pi$ in the unitary dual $\widehat{G}$
of $G$, there exists $\ell\in\mathfrak{g}^{*}$ and a subalgebra $\mathfrak{p}
\subset\mathfrak{g}$, subordinate to $\ell$, such that $\pi$ is unitarily
equivalent to $\pi_{\ell} = \pi(\ell, \mathfrak{p})$.
A representation $\pi_{\ell} = \pi(\ell,
\mathfrak{p})$, with $\mathfrak{p}$
subordinate to $\ell \in \mathfrak{g}^*$, is irreducible if, and only if, $\mathfrak{p}$ is a maximal subalgebra subordinated to
$\ell\in\mathfrak{g}^{*}$ satisfying
$
\dim\left(  \mathfrak{p}\right)  =\dim\left(  \mathfrak{g}\right)  - \dim( \mathcal{O}_{\ell}) / 2,
$ a so-called \emph{(real) polarisation}.

Two irreducible induced representations $\mathrm{ind}_{\exp (
\mathfrak{p})  }^{G}(  \chi_{\ell})  $ and $\mathrm{ind}
_{\exp(  \mathfrak{p}^{\prime})  }^{G}(  \chi_{\ell^{\prime}
})  $ are unitarily equivalent if and only if the linear functionals
$\ell,\ell^{\prime} \in\mathfrak{g}^{*}$ belong to the same coadjoint orbit.
The orbit associated to the equivalence class $\pi\in\widehat{G}$ will also be
denoted by $\mathcal{O}_{\pi}$.

\subsection{Moment set}
Let $(\pi, \Hpi)$ be an irreducible unitary representation of $G$.
Denote by $\Hs$ the space of smooth vectors for $\pi$, i.e. the space of $\eta \in \Hpi$ for which $g \mapsto \pi(g) \eta$ is smooth.

The derived representation $d\pi : \mathfrak{g} \to L(\Hpi^{\infty})$ is defined by
\begin{align} \label{eq:derived_rep}
d\pi (X) \eta = \frac{d}{dt} \bigg|_{t = 0} \pi (\exp (tX)) \eta, \quad X \in \mathfrak{g}, \;\eta \in \Hs.
\end{align}
It can be extended complex linearly to a representation of the complexification $\mathfrak{g}_{\mathbb{C}}$ of $\mathfrak{g}$.

The  \emph{moment map} of $\pi$ is the mapping $J_{\pi} : \Hs \to \mathfrak{g}^*$ defined by
\begin{align} \label{eq:moment_map}
J_{\pi} (\eta)(X) = \frac{1}{i} \frac{\langle d\pi (X) \eta, \eta \rangle}{\langle \eta, \eta \rangle}, \quad X \in \mathfrak{g}, \; \eta \in \Hs.
\end{align}
Note that the right-hand side of \eqref{eq:moment_map} only depends on the ray $[\eta]$ generated by $\eta \in \Hs \setminus \{0\}$.

The moment map $J_{\pi}$ is equivariant with respect to the canonical $G$-actions on $\Hs$ and $\mathfrak{g}^*$, i.e.
$J_{\pi} (\pi (g) \eta ) (X) = (\capad(g)^* J_{\pi} (\eta)) (X)$ for $g \in G$, $X \in \mathfrak{g}$ and $\eta \in \Hs $.
In particular,  $J_{\pi} (G \cdot \eta)$ is the coadjoint orbit $\mathcal{O}_{J_{\pi} (\eta)}$ of $J_{\pi} (\eta) \in \mathfrak{g}^*$.

The \emph{moment set} $I_{\pi}$ of $\pi$ is the closure $I_{\pi} := \overline{J_{\pi} (\Hs)}$ in $\mathfrak{g}^*$. Its relation to the coadjoint $\mathcal{O}_{\pi}$ of $\pi \in \widehat{G}$ is
\begin{align} \label{eq:momentset}
 I_{\pi} = \overline{\conv} (\mathcal{O}_{\pi}),
\end{align}
where $\overline{\conv}$ denotes the closed convex hull; see \cite[Theorem 4.2]{wildberger1989convexity}.

\subsection{Coherent state representations} \label{sec:coherent_state_rep}
Henceforth, it is assumed that $(\pi, \Hpi)$ is non-trivial.
Let $\eta \in \Hpi$ be a unit vector and let $H \leq G$ be a closed subgroup such that there exists a unitary character $\chi : H \to \mathbb{T}$ satisfying
\begin{align} \label{eq:stable_character}
\pi (h) \eta = \chi (h) \eta, \quad h \in H.
\end{align}
Denote $X := G/H$  and let $\mu_X$ be $G$-invariant Radon measure on $X$, which is unique up to scalar multiplication. Fix a Borel cross-section $\sigma : X \to G$ for the quotient map $p : G \to X$.
The vector $\eta$ is called \emph{admissible}
if \begin{align} \label{eq:admissible}
\int_X | \langle \eta, \pi (\sigma(x)) \eta \rangle |^2 \; d\mu_X (x) < \infty.
\end{align}
A pair $(\eta, \chi)$ satisfying \eqref{eq:stable_character} and \eqref{eq:admissible} is said to define a \emph{$\pi$-system of coherent states} based on $X = G / H$. The condition \eqref{eq:admissible} is independent of the particular choice of section $\sigma$.

For a $\pi$-system of coherent states, there exists an \emph{admissibility constant} $d_{\pi, \eta} > 0$ such that, for all $f \in \Hpi$,
\begin{align} \label{eq:orthogonality}
\int_X | \langle f, \pi (\sigma(x)) \eta \rangle |^2 \; d\mu_X (x) = d_{\pi, \eta}^{-1} \| f \|_{\Hpi}^2.
\end{align}
For further properties on square-integrability modulo a subgroup, see, e.g. \cite{neeb1997square, moscovici1978coherent}.

An irreducible representation $(\pi, \Hpi)$ is called a \emph{coherent state representation} if it admits a $\pi$-system of coherent states based on connected, simply connected homogeneous $G$-space $X$.\footnote{The definition of a coherent state representation used here is the same as in \cite{neeb1997square, moscovici1977coherent, moscovici1978coherent}, but differs from the definition in \cite{neeb1996coherent, lisiecki1991classification, lisiecki1990kaehler}, where the square-integrability assumption \eqref{eq:admissible} is not part of the definition.}

\section{Completeness of coherent state subsystems}
This section considers the relation between subsystems of coherent states parametrised by a simply connected $G$-space and lattice orbits of an associated projective representation.

\subsection{Projective kernel}
The \emph{kernel} and \emph{projective kernel} of a unitary representation $(\pi, \Hpi)$ of $G$ are defined by
\[
\ker(\pi) = \{ g \in G \; : \; \pi (g) =  I_{\Hpi} \} \quad \text{and} \quad \pker(\pi) = \{ g \in G \; : \; \pi (g) \in \mathbb{C} \cdot I_{\Hpi} \},
\]
respectively. If $(\pi, \Hpi)$ is non-trivial and irreducible, then $\pker(\pi) \leq G$ is a connected, closed normal subgroup, and there exists $\chi_{\pi} : \pker(\pi) \to \mathbb{T}$ such that $\pi(g) = \chi_{\pi} (g) I_{\Hpi}$ for $g \in \pker(\pi)$.

The following observation plays a key role in the sequel. Its proof is a combination of results in \cite{moscovici1977coherent}, which characterise coherent state representations $\pi$ in terms of their coadjoint orbit.

\begin{proposition} \label{prop:isotropy_kernel}
Let $H \leq G$ be a connected subgroup. Suppose $\pi$ admits a $\pi$-system of coherent states based on $G / H$.
 Then
$H = \pker(\pi)$. In particular,  $H \leq G$ is normal.

\end{proposition}
\begin{proof}
If $\pi$ admits a pair $(\eta, \chi)$ satisfying \eqref{eq:stable_character} and \eqref{eq:admissible}, then $\pi$ is unitarily equivalent to a subrepresentation of the induced representation $\ind_{H}^G \chi$, see, e.g. \cite[Proposition 1.2]{moscovici1977coherent}. Since $H \leq G$ is assumed to be connected, it follows by \cite[Lemma 3.5]{moscovici1977coherent} that $H = G(\ell)$ for any $\ell \in \mathcal{O}_{\pi}$. By \cite[Theorem 2.1]{bekka1990complemented}, the projective kernel of an arbitrary irreducible representation $\pi$ of $G$  is given by $\pker(\pi) = \bigcap_{\ell \in \mathcal{O}_{\pi}} G(\ell)$. Therefore, $\pker(\pi) = \bigcap_{\ell \in \mathcal{O}_{\pi}} G(\ell) = H$.
\end{proof}

The conclusion of Proposition \ref{prop:isotropy_kernel} may fail for disconnected subgroups $H \leq G$ whenever $\pi$ has a discrete kernel:

\begin{remark} \label{rem:SI}
Let $(\pi, \Hpi)$ be an irreducible unitary representation of $G$.
\begin{enumerate}[(a)]
\item If $\pi$ is square-integrable modulo $K = \ker(\pi)$, then $\pi|_{K}$ satisfies \eqref{eq:stable_character} for the trivial character $\chi \equiv 1$ and any vector $\eta \in \Hpi$ defines a $\pi$-system of coherent states based on $G/K$.
\item If $\pi$ is square-integrable modulo $Z = Z(G)$, then $\pi|_{Z}$ satisfies \eqref{eq:stable_character} for
the central character $\chi \in \widehat{Z}$ and any vector $\eta \in \Hpi$
defines a $\pi$-system of coherent states based on $G/Z$. Moreover, $\pker(\pi) = Z(G)$ by \cite[Corollary 4.5.4]{corwin1990representations}.
\end{enumerate}
\end{remark}

\subsection{Necessary density conditions}
A \emph{uniform subgroup} $\Gamma \leq G$ is a discrete subgroup such that $\Gamma \backslash G$ is compact. For a nilpotent Lie group $G$, the uniformity of a discrete subgroup $\Gamma \leq G$ is equivalent to $\Gamma$ having finite co-volume, see, e.g. \cite[Corollary 5.4.6]{corwin1990representations}.

The following result provides a criterium for cyclicity of restricted (projective) representations in terms of the lattice co-volume or density  (cf. \cite[Theorem 7.4]{romero2020density}).

\begin{theorem}[\cite{romero2020density}] \label{thm:density_lattices}
Let $(\pi, \Hpi)$ be an irreducible, square-integrable projective unitary representation of a unimodular group $G$, with formal dimension $d_{\pi} > 0$.  Let $\Gamma \leq G$ be a lattice. If there exists $\eta \in \Hpi$ such that $\pi (\Gamma) \eta$ is complete in $\Hpi$, then $\covol(\Gamma) d_{\pi} \leq 1$.
\end{theorem}

For a genuine representation $\pi$ of  $G$ that is square-integrable modulo the centre $Z(G)$, a version of Theorem \ref{thm:density_lattices} can also be deduced from \cite[Theorem 5]{bekka2004square}; see also \cite[Theorem 3]{bekka2004square} for a converse in the setting of nilpotent Lie groups. However, in order to treat a representation $\pi$ that is merely square-integrable modulo $ \ker(\pi)$ (equivalently, $\pker(\pi)$), the projective version of Theorem \ref{thm:density_lattices} is particularly convenient for the purposes of the present note.

The following completeness result for coherent state subsystems can simply be obtained by combining Proposition \ref{prop:isotropy_kernel} and Theorem \ref{thm:density_lattices}.

\begin{theorem} \label{thm:subsystems_states}
Let $H \leq G$ be a connected subgroup.
Suppose $(\pi, \Hpi)$ is an irreducible representation that admits an admissible vector $\eta \in \Hpi$ defining a $\pi$-system of coherent states
based on $X = G/H$, with admissibility constant $d_{\pi, \eta} > 0$. Then
\begin{enumerate}
 \item[(i)] $H = \pker(\pi)$;
 \item[(ii)] If $\Gamma \leq G$ is uniform and  $\{ \pi (\sigma (\gamma' )) \eta \}_{\gamma' \in \Gamma \cdot o}$ is complete,
then $\covol(p(\Gamma)) d_{\pi, \eta} \leq 1$.
\end{enumerate}
\end{theorem}
\begin{proof}
By Proposition \ref{prop:isotropy_kernel}, the admissibility of $\pi$ implies that $H = \pker(\pi) \leq G$ is normal.
Hence, the induced mapping
$
\pi' : G/H \to \mathcal{U}(\Hpi), \; x \mapsto \pi (\sigma (x))
$
forms an irreducible projective representation of $G/H$. Since the measure $\mu_{X}$ is Haar measure on $X = G/ H$, it follows that $\pi'$ is square-integrable on $G/H$ by the admissibility condition \eqref{eq:admissible}. In particular, the constant $d_{\pi, \eta} > 0$ in \eqref{eq:orthogonality} coincides with the (unique) formal dimension $d_{\pi'} > 0$ of the projective representation $(\pi', \Hpi)$ normalised according to the $G$-invariant measure $\mu_X$.

Suppose $\Gamma \leq G$ is a uniform subgroup. As in the proof of Proposition \ref{prop:isotropy_kernel}, the admissibility of $\pi$ implies that  $\pker(\pi) = G(\ell)$ for any $\ell \in \mathcal{O}_{\pi}$. A combination of \cite[Proposition 5.2.6]{corwin1990representations} and \cite[Theorem 5.1.11]{corwin1990representations} therefore yields that $\Gamma \cap H$ is a uniform subgroup of $H = \pker(\pi)$. Hence, the image $p(\Gamma)$ is a uniform subgroup of $G/H$ by \cite[Lemma 5.1.4(a)]{corwin1990representations}.

In combination, applying Theorem \ref{thm:density_lattices} to $(\pi', \Hpi)$ and $p(\Gamma) \leq G/H$ yields the result.
\end{proof}

\begin{remark}
The constant $d_{\pi, \eta} > 0$ coincides with the formal dimension $d_{\pi'} > 0$ of the projective representation $(\pi', \Hpi)$ of $X = G/\pker(\pi)$. In particular, the product
$\covol(p(\Gamma)) d_{\pi'}$ is independent of the choice of $G$-invariant measure $\mu_X$: if $\mu'_X = c \cdot \mu_X$ for $c > 0$, then $\covol'(p(\Gamma)) = c \cdot \covol(p(\Gamma))$ and $d'_{\pi'} = d_{\pi'} / c$.
\end{remark}

Theorem \ref{thm:intro} follows directly from Proposition \ref{prop:isotropy_kernel} and Theorem \ref{thm:subsystems_states}:

\begin{proof}[Proof of Theorem \ref{thm:intro}]
By assumption, there exists an admissible $\eta \in \Hpi$ and associated character $\chi : H \to \mathbb{T}$ defining a $\pi$-system of coherent states based on  $G / H \cong \mathcal{O}_{\pi}$.
Since $\mathcal{O}_{\pi}$ is simply connected, it follows that $H \subset G$ is connected, see, e.g. \cite[Proposition 1.94]{knapp2002lie}. The conclusions are therefore a direct consequence of Proposition \ref{prop:isotropy_kernel} and Theorem \ref{thm:subsystems_states}.
\end{proof}

\section{Perelomov-type coherent states} \label{sec:perelomov}
Let $(\pi, \Hpi)$ be an irreducible representation of $G$.
Then $\pi$ yields an action of $G$ on the projective spaces $\PH$ and $\PHs$ by $g \cdot [\eta] = [\pi(g) \eta]$.

A system of \emph{Perelomov-type coherent states} is a $G$-orbit in $\PH$,
\[ G \cdot [\eta] = \big \{ [ \pi(g) \eta ] \; : \; g \in G \big\} .\]
Let
$G_{[\eta]}$ be the isotropy group of $\eta \in \Hpi \setminus \{0\}$ in the projective space $\PH$,
\begin{align} \label{eq:isotropy_projective}
 G_{[\eta]} := \big\{ g \in G \; : \; \pi (g) \eta \in \mathbb{C} \eta \big\}.
\end{align}
Denote by $X = G/G_{[\eta]}$ the associated homogeneous space and let $\sigma : X \to G$ be a Borel section for the quotient map $p : G \to X$. Then a Perelomov-type coherent state system is determined
by the system of vectors,
\[
\{ \eta_x \}_{x \in X} = \{ \pi (\sigma(x)) \eta \}_{x \in X}.
\]
See \cite[Section 2]{perelomov1972coherent} and \cite[Chapter 2]{perelomov1986generalized} for basic properties of Perelomov-type states.

Let $\chi_{\eta} : G_{[\eta]} \to \mathbb{T}$ be the unitary character of $ G_{[\eta]}$ such that $\pi(g) \eta = \chi_{\eta} (g) \eta$ for all $g \in G_{[\eta]}$. Note that $G_{[\eta]}$ is the maximal subgroup satisfying the property \eqref{eq:stable_character} for a chosen $\eta$.

The following sections consider Perelomov-type coherent states of vectors $\eta \in \Hs \setminus \{0\}$ with the property that $G \cdot [\eta]$ has a symplectic or complex structure. Such coherent states are of particular interest for geometric quantization, see \cite{Odzijewicz1992coherent} and \cite[Section 16]{perelomov1986generalized}.

\subsection{Symplectic projective orbits} \label{sec:symplectic_orbit}
Following \cite{lisiecki1990kaehler, kostant1982symplectic}, an orbit $G \cdot [\eta] = \{ [\pi(g) \eta] : g \in G \}$ is called \emph{symplectic} if $[\eta] \in \PHs$ and $G \cdot [\eta]$ is a symplectic submanifold of $\PH$.

The following characterisation of symplectic orbits will be used below; see \cite[Theorem 26.8]{guillemin1984symplectic} for its proof.

\begin{lemma}[\cite{guillemin1984symplectic}] \label{lem:symplectic_orbit}
Let $[\eta] \in \PHs$ and let $J_{\pi} : \PHs \to \mathfrak{g}^*$ be the momentum map of $\pi$.
The orbit $G \cdot [\eta]$ is symplectic if, and only if, the stabiliser $G_{[\eta]}$ is an open subgroup of $G( J_{\pi} ([\eta]))$.
\end{lemma}

For the purposes of this note, the significance of a symplectic orbit is that its stabiliser subgroup coincides with the projective kernel, and hence does not depend on the chosen vector.
This is demonstrated by the following proposition.

\begin{proposition} \label{prop:connected_isotropy}
Suppose $\eta \in \Hpi^{\infty} \setminus \{0\}$ is such that $G \cdot [\eta]$ is symplectic. Then $G_{[\eta]}$ is connected.
In particular, if $\eta$ is an admissible vector defining a $\pi$-system of coherent states based on $G / G_{[\eta]}$, then $G_{[\eta]} = \pker(\pi)$.
\end{proposition}
\begin{proof}
If $G \cdot [\eta]$ is symplectic, then $G \cdot [\eta]$ forms a Hamiltonian $G$-space, with momentum map $J_{\pi} : G \cdot [\eta] \to \mathfrak{g}^*$ given as in \eqref{eq:moment_map}, see, e.g. \cite[Section 2.5]{lisiecki1990kaehler}. Set $\ell := J_{\pi} ([\eta])$. Then, by Lemma \ref{lem:symplectic_orbit},
the stabiliser $G_{[\eta]}$ is an open subgroup of $G (\ell)$.
Since $G(\ell)$ is connected (cf. Section \ref{sec:coadjoint}),
it follows that $G_{[\eta]} = G(\ell)$ is connected. The last assertion follows from Proposition \ref{prop:isotropy_kernel}.
\end{proof}

The following provides a partial converse to Proposition \ref{prop:connected_isotropy}.

\begin{proposition} \label{prop:SI_symplectic}
 Suppose $(\pi, \Hpi)$ is square-integrable modulo $\pker(\pi)$. Then, for any $[\eta] \in \PHs$, the orbit $G \cdot [\eta]$ is symplectic and $G_{[\eta]} = \pker(\pi)$.
\end{proposition}
\begin{proof}
Let $\eta \in \Hs \setminus \{0\}$ be fixed. The inclusion $\pker(\pi) \subseteq G_{[\eta]}$ is immediate. Conversely,
if $g \in G_{[\eta]}$, then
\[
J_{\pi} ([\pi(g) \eta]) (X) = \frac{1}{i} \frac{\langle \pi(g) \eta,  d \pi (X) \pi(g) \eta \rangle}{\langle \pi(g) \eta , \pi (g) \eta \rangle } = \frac{1}{i} \frac{\langle  \eta,  d \pi (X)  \eta \rangle}{\langle \eta, \eta \rangle} = J_{\pi} ([\eta]) (X), \quad X \in \mathfrak{g},
\]
so that by the $G$-equivariance of $J_{\pi}$ it follows that $\capad^* (g) J_{\pi} ([\eta]) = J_{\pi} ([\eta])$. This means that $g \in G(J_{\pi} ([\eta]))$,
and it remains to show that $G(J_{\pi} ([\eta])) \subseteq \pker(\pi)$.

Since $\pi \in \widehat{G}$ is square-integrable modulo $\pker(\pi)$,
it is also square-integrable modulo $\ker(\pi)$, see, e.g. \cite[Corollary 2.1]{bekka1990complemented}.
It follows therefore by \cite[Theorem 4.5.2]{corwin1990representations} and \cite[Theorem 3.2.3]{corwin1990representations} that $\mathcal{O}_{\pi}$ is a linear variety of the form $\mathcal{O}_{\pi}  = \ell + \mathfrak{k}^{\perp}$ for $\ell \in \mathcal{O}_{\pi}$, with $\mathfrak{k}$ being the Lie algebra of $\pker(\pi)$. In addition, \cite[Theorem 3.2.3]{corwin1990representations} yields that
$\mathfrak{g} (\ell) = \mathfrak{k}$ for $\ell \in \mathcal{O}_{\pi}$, so that $G (\ell) = \pker(\pi)$ for $\ell \in \mathcal{O}_{\pi}$. By \cite[Theorem 4.2]{wildberger1989convexity} (see also Equation \eqref{eq:momentset}) it follows, in particular, that
\[
 J_{\pi} ([\eta]) \in J_{\pi} (\PHs) \subseteq I_{\pi} = \overline{\conv}(\mathcal{O}_{\pi}) = \mathcal{O}_{\pi},
\]
where $I_{\pi} = \overline{J_{\pi} (\PHs)}$ denotes the moment set of $\pi$. Therefore, $G(J_{\pi}([\eta])) = \pker(\pi)$.

Lastly, since $G_{[\eta]} = \pker(\pi) = G(J_{\pi} ([\eta]))$ by the above, the orbit $G \cdot [\eta]$ is symplectic by Lemma \ref{lem:symplectic_orbit}.
\end{proof}

\begin{proof}[Proof of Theorem \ref{thm:intro2}]
If  $G \cdot [\eta]$ is symplectic, then $G_{[\eta]}$ is connected by Proposition \ref{prop:connected_isotropy}. Therefore, if $\eta$ determines a $\pi$-system of coherent states based on $G/G_{[\eta]}$, the conclusions of Theorem \ref{thm:intro2} follow directly from Theorem \ref{thm:subsystems_states}.
\end{proof}

\subsection{Complex projective orbits} \label{sec:complex_orbit}
In \cite{lisiecki1990kaehler, neeb1996coherent}, an orbit $G \cdot [\eta] = \{ [\pi(g) \eta] : g \in G \}$ is called \emph{complex} if $[\eta] \in \PHs$ and $G \cdot [\eta]$ is a complex submanifold of $\PH$.

The following lemma characterises complex orbits in terms of the stabiliser of the complexified Lie algebra; see \cite[Proposition 2.8]{lisiecki1990kaehler} and \cite[Lemma XV.2.3]{neeb1999holomorphy}.

\begin{lemma}[\cite{lisiecki1990kaehler}] \label{lem:complex_orbit}
Let $\mathfrak{s} = (\mathfrak{g})_{\mathbb{C}}$. For $[\eta] \in \PHs$, let $\mathfrak{s}_{[\eta]} := \{ X \in \mathfrak{s} : d\pi(X) \eta \in \mathbb{C} \eta \}$.

The following assertions are equivalent:
\begin{enumerate}[(i)]
 \item The orbit $G \cdot [\eta]$ is complex;
 \item $\mathfrak{s}_{[\eta]} + \overline{\mathfrak{s}_{[\eta]}} = \mathfrak{s}$.
\end{enumerate}

\end{lemma}

A stabiliser $\mathfrak{s}_{[\eta]}$ satisfying part (ii) of Lemma \ref{lem:complex_orbit} is called \emph{maximal} in \cite[Section 2.4]{perelomov1986generalized}, where it is part of a principle for selecting coherent states that minimise the uncertainty principle.
Such vectors and associated orbits play an important role in Berezin's quantization, see \cite[Section 16]{perelomov1986generalized}. In addition, vectors of this type are intimately related to highest weight modules and representations (cf. \cite{neeb1996coherent, neeb1999holomorphy}) and are also referred to as \emph{highest weight vectors}.

\begin{remark}  \label{rem:complex_orbit}
By \cite[Proposition 2.8]{lisiecki1990kaehler}, any complex orbit is automatically symplectic in the sense of Section \ref{sec:symplectic_orbit}. Theorem \ref{thm:intro2} applies therefore to highest weight vectors.
\end{remark}

\begin{remark}
The significance of a complex orbit $G \cdot [\eta]$ is that the quotient manifold $G / G_{[\eta]}$ admits a complex structure (cf. \cite[Section XV.2]{neeb1999holomorphy}). In turn, for certain (classes of) representations admitting highest weight vectors, the representation space may be realised as a space of holomorphic functions (see \cite[Section 2.4]{perelomov1986generalized} and \cite{rossi1973representations}); in particular, see \cite[Section 5]{lisiecki1995coherent} for complex orbits for the Heisenberg group. However, for nilpotent Lie groups, the existence of complex orbits appears to be restrictive, i.e. \cite[Theorem 1]{lisiecki1991classification} asserts that the only irreducible representations with a discrete kernel admitting complex orbits are those of Heisenberg groups. In contrast, symplectic orbits do exist for all groups admitting square-integrable representations by Proposition \ref{prop:SI_symplectic}.
\end{remark}

\section*{Acknowledgements}
J.v.V. gratefully acknowledges support from the Research
Foundation - Flanders (FWO) Odysseus 1 grant G.0H94.18N and the Austrian Science Fund (FWF) project J-4445.
Thanks are due to Bas Janssens for helpful discussions and to the anonymous referee for providing helpful comments and suggestions.

\end{document}